\documentclass[12pt,reqno]{amsart}
\usepackage{amsmath, amsthm, amssymb}

\advance \topmargin by -\headheight
\advance \topmargin by -\headsep
     
\evensidemargin \oddsidemargin
\newtheorem{theorem}{Theorem}[section]
\newtheorem{lemma}[theorem]{Lemma}

\theoremstyle{definition}

\theoremstyle{remark}

\numberwithin{equation}{section}

\def\bfa{{\mathbf a}}

\def\bfc{{\mathbf c}}
 
\def\bfe{{\mathbf e}} 

\def\bfh{{\mathbf h}} \def\bfH{\mathbf H}

\def\bfm{{\mathbf m}}
\def\bfn{{\mathbf n}}

\def\bfv{{\mathbf v}}

 \def\bfX{{\mathbf X}}

\def\calA{{\mathcal A}}  

\def\calB{{\mathcal B}} 
\def\calC{{\mathcal C}}

\def\calH{{\mathcal H}}

\def\calN{{\mathcal N}}

\def\calP{{\mathcal P}}

\def\calS{{\mathcal S}}

\def\dbC{{\mathbb C}}\def\dbN{{\mathbb N}}
\def\dbR{{\mathbb R}}
\def\dbZ{{\mathbb Z}}

\def\grB{{\mathfrak B}}

\def\grf{{\mathfrak f}}\def\grF{{\mathfrak F}}
\def\grg{{\mathfrak g}}\def\grG{{\mathfrak G}}

\def\grJ{{\mathfrak J}}

\def\grm{{\mathfrak m}}\def\grM{{\mathfrak M}}\def\grN{{\mathfrak N}}
\def\grn{{\mathfrak n}}\def\grS{{\mathfrak S}}\def\grP{{\mathfrak P}}
\def\grB{{\mathfrak B}}

\def\grv{{\mathfrak v}}\def\grV{{\mathfrak V}}

\def\alp{{\alpha}} \def\bfalp{{\boldsymbol \alpha}} 
\def\bfalphat{\widehat{\bfalp}}
\def\bet{{\beta}}  \def\bfbet{{\boldsymbol \beta}}
\def\gam{{\gamma}}  
 \def\bfgam{{\boldsymbol \gamma}}

\def\del{{\delta}} \def\Del{{\Delta}}

\def\tet{{\theta}} \def\bftet{{\boldsymbol \theta}} 
\def\bftethat{{\widehat \bftet}} 
\def\kap{{\kappa}}
\def\lam{{\lambda}} \def\bflam{{\boldsymbol \lam}}
\def\Lam{{\Lambda}} 
\def\bfnu{{\boldsymbol \nu}}

\def\sig{{\sigma}}  

\def\Ups{{\Upsilon}} 
\def\bfphi{{\boldsymbol \phi}} \def\bfPhi{{\boldsymbol \Phi}}

\def\ome{{\omega}} \def\Ome{{\Omega}} 
\def\d{{\partial}}
\def\eps{\varepsilon}

\def\le{\leqslant} \def\ge{\geqslant}

\def\d{{\,{\rm d}}}

\begin{document}
\title[Sums of three cubes]{Correlation estimates for sums of three cubes}
\author[J\"org Br\"udern]{J\"org Br\"udern}
\address{Mathematisches Institut, Bunsenstrasse 3--5, D-37073 G\"ottingen, Germany}
\email{bruedern@uni-math.gwdg.de}
\author[Trevor D. Wooley]{Trevor D. Wooley}
\address{School of Mathematics, University of Bristol, University Walk, Clifton, Bristol BS8 1TW, United 
Kingdom}
\email{matdw@bristol.ac.uk}
\thanks{The authors thank the Hausdorff Institute in Bonn, the Heilbronn Institute in 
Bristol, and CIRM in Luminy, for excellent working conditions conducive to writing this 
paper.}
\date{}
\begin{abstract} We establish estimates for linear correlation sums involving sums of 
three positive integral cubes. Under appropriate conditions, the underlying methods permit 
us to establish the solubility of systems of homogeneous linear equations in sums of three 
positive cubes whenever these systems have more than twice as many variables as 
equations.
\end{abstract}
\maketitle

\section{Introduction} We shall be concerned in this memoir with the number $\rho(n)$ of 
ways the natural number $n$ can be written as the sum of three positive integral cubes. 
Our principal goal is to provide upper bounds for linear correlation sums involving 
$\rho(n)$ and certain of its relatives. As an application of the underlying methods, we 
consider the solubility of systems of homogeneous linear equations in sums of three 
positive integral cubes. Provided that the system is in general position, and it has a 
solution in positive integers, we are able to show that it is soluble in sums of three positive 
cubes whenever the number of variables exceeds twice the number of equations.\par

Some notation is required before we may introduce the family of higher correlation sums 
that are central to our focus. Let $s$ and $r$ be natural numbers with $s\ge r$, and 
consider an $r\times s$ integral matrix $A=(a_{ij})$. We associate with $A$ the collection 
of linear forms
\begin{equation}\label{1.1}
\Lam_j(\bfalp)=\sum_{i=1}^ra_{ij}\alp_i\quad (1\le j\le s),
\end{equation}
and its positive cone
$$\calP=\{ \bfalp\in \dbR^r: \text{$\alp_i>0$ $(1\le i\le r)$ and $\Lam_j(\bfalp)>0$ 
$(1\le j\le s$)}\}.$$Note that $\calP$ is open, and hence its truncation 
$\calP(N)=\calP\cap [1,N]^r$ has measure $\gg N^r$ whenever $\calP$ is non-empty. 
Given an $s$-tuple $\bfh$ of non-negative integers, we may now define 
the sum $\Xi_s(N)=\Xi_s(N;A;\bfh)$ by putting
\begin{equation}\label{1.2}
\Xi_s(N;A;\bfh)=\sum_{\bfn\in \calP(N)} \rho(\Lam_1(\bfn)+h_1)\cdots 
\rho(\Lam_s(\bfn)+h_s).
\end{equation}
We refer to the coefficient matrix $A$ as being {\it highly non-singular} if all collections of 
at most $r$ of its $s$ columns are linearly independent.

\begin{theorem}\label{theorem1.1} Let $A\in \dbZ^{r\times 2r}$ be highly non-singular, 
and let $h_i\in \dbN\cup \{ 0\}$ $(1\le i\le 2r)$. Then 
$\Xi_{2r}(N;A;\bfh)\ll N^{r+1/6+\eps}$, where the constant implict in Vinogradov's 
notation depends at most on $A$ and $\eps>0$.
\end{theorem}

Classical approaches to the simplest correlation sum proceed via Cauchy's inequality. Thus, 
by utilising Hua's lemma (see \cite[Lemma 2.5]{Vau1997}), one obtains
\begin{equation}\label{1.3}
\sum_{n\le N}\rho(n)\rho(n+h)\ll \sum_{n\le N+h}\rho(n)^2\ll N^{7/6+\eps}.
\end{equation}
This traditional argument is easily generalised to handle the sum $\Xi_{2r}(N)$. Writing 
$m_j=\Lam_j(\bfn)+h_j$ for the sake of brevity, Cauchy's inequality yields
$$\Xi_{2r}(N)\le \prod_{j\in\{0,1\}}\biggl( 
\sum_{\bfn\in \calP(N)}\rho(m_{jr+1})^2\cdots \rho(m_{jr+r})^2\biggr)^{1/2}.$$
Since $\Lam_1,\Lam_2,\ldots ,\Lam_r$ are linearly independent, one may sum over the 
values $m_1,m_2,\ldots ,m_r$ as if these were independent variables. Thus, by symmetry, 
it follows as a consequence of the second inequality of (\ref{1.3}) that there is a number 
$C=C(A)\ge 1$ such that
\begin{equation}\label{1.4}
\Xi_{2r}(N)\le \biggl( \sum_{n\le CN}\rho(n)^2\biggr)^r\ll N^{7r/6+\eps}.
\end{equation}

\par The bound (\ref{1.4}) is certainly part of the folklore in the area, and constitutes the 
state of the art hitherto. It is widely believed that the upper bound $N^{7/6+\eps}$ in 
(\ref{1.3}) may be replaced by $N$, and indeed the slightly weaker estimate 
$N^{1+\eps}$ has been established by Hooley \cite{Hoo1997} and Heath-Brown 
\cite{HB1998} based on speculative hypotheses concerning the distribution of the zeros of 
certain Hasse-Weil $L$-functions. Accepting one or other of these estimates as a working 
hypothesis, one finds that $\Xi_{2r}(N)\ll N^{r+\eps}$, or even $\Xi_{2r}(N)\ll N^r$. For 
certain coefficient matrices $A$, readers will have little difficulty in convincing themselves 
that the lower bound $\Xi_{2r}(N)\gg N^r$ is to be expected. Although the bound on 
$\Xi_{2r}(N)$ presented in Theorem \ref{theorem1.1} does not improve on the classical 
estimate (\ref{1.4}) when $r=1$, for all larger values of $r$ it is substantially sharper.\par

For applications to problems of Waring's type, mollified versions of $\rho(n)$ have been 
utilised since the invention by Hardy and Littlewood \cite{HL1925} of diminishing ranges. 
Most modern innovations within this circle of ideas involve the use of sets of smooth 
numbers having positive density. Thus, given $\eta>0$, let $\rho_\eta(n)$ denote the 
number of integral solutions of the equation $n=x^3+y^3+z^3$, subject to the condition 
that none of the prime divisors of $yz$ exceed $n^{\eta/3}$. Then it follows from 
\cite{Woo1995, Woo2000} that for each $\eps>0$, there is a positive number $\eta$ such 
that
\begin{equation}\label{1.Z}
\sum_{1\le n\le N}\rho_\eta(n)^2\ll N^{1+\xi+\eps},
\end{equation}
where $\xi=(\sqrt{2833}-43)/123<1/12$. Define $\Xi_{s,\eta}(N;A;\bfh)$ as in 
(\ref{1.2}), but with $\rho_\eta$ in place of $\rho$ throughout.

\begin{theorem}\label{theorem1.2}
Let $A\in \dbZ^{r\times 2r}$ be highly non-singular, and let $h_i\in \dbN\cup \{0\}$ 
$(1\le i\le 2r)$. Then for each $\eps>0$, there is a number $\eta>0$ such that
$$\Xi_{2r,\eta}(N;A;\bfh)\ll N^{r+\xi+\eps}.$$
The constant implict in Vinogradov's notation depends at most on $A$, $\eps$ and $\eta$.
\end{theorem}

We turn now to systems of linear equations in sums of three cubes. Let 
$C\in \dbZ^{r\times s}$ be highly non-singular, and suppose that the system
\begin{equation}\label{1.6}
\sum_{j=1}^sc_{ij}n_j=0\quad (1\le i\le r)
\end{equation}
has a solution in positive integers $n_1,\ldots ,n_s$. Denote by $\Ups(N)$ the number of 
solutions of the system (\ref{1.6}) with $n_j\le N$ in which $n_j$ is a sum of three 
positive integral cubes. We emphasise that $\Ups(N)$ counts solutions without weighting 
them for the number of representations as the sum of three cubes.

\begin{theorem}\label{theorem1.3}
Let $C\in \dbZ^{r\times s}$ be highly non-singular, and suppose that {\rm (\ref{1.6})} 
has a solution $\bfn\in (0,\infty)^s$. Then whenever $s>2r$ and $\eps>0$, one has
$$\Ups(N)\gg N^{s(1-2\xi)-r-\eps}.$$
\end{theorem}

Were sums of three positive integral cubes to have positive density in the natural 
numbers, then one imagines that a suitable enhancement of the methods of Gowers 
\cite{Gow1998} ought to deliver the stronger conclusion $\Ups(N)\gg N^{s-r}$ for 
$s\ge r+2$. However, there seems to be no prospect of any such density result at 
present, and so one is forced to contemplate the possibility that the number of positive 
integers $n\le N$, representable as the sum of three positive integral cubes, may be as 
small as $N^{1-\xi}$. In such circumstances, even the lower bound $\Ups(N)\ge 1$ is 
highly non-trivial. Indeed, in cases where $s$ is close to $2r+1$, such a conclusion is 
established for the first time within this paper. When sums of three cubes are replaced by 
sums of two squares, on the other hand, the value set comes very close to achieving 
positive density, and the methods of Gowers are in play. In this setting, the work of 
Matthiesen \cite{Mat2012, Mat2013} comes within a factor $N^\eps$ of achieving the 
natural analogue of the above lower bound.\par

Subject to appropriate additional hypotheses, a conclusion similar to that of Theorem 
\ref{theorem1.3} may be obtained for the analogue of $\Ups(N)$ in which (\ref{1.6}) is 
replaced by an inhomogeneous system of linear equations. Note also that Balog and 
Br\"udern \cite{BB1995} consider systems of linear equations in sums of three cubes of 
special type. In the case of a single equation, their work more efficiently removes the 
multiplicity inherent in $\rho_\eta(n)$, and establishes a superior bound in this case for 
$\Ups(N)$.\par

The conclusions of this paper depend on a new mean value estimate that is of 
independent interest. In \S2 we examine systems of equations in which the coefficient 
matrices are of linked block type, and establish an auxiliary bound for their number of 
solutions. This prepares the way for the proof of the central estimate, Theorem 
\ref{theorem3.3}, in \S3, accomplished by a novel {\it complification} argument in which 
mean values are bounded by blowing up the number of equations so as to apply the 
powerful estimates of the previous section. We then establish the correlation estimates of 
Theorems \ref{theorem1.1} and \ref{theorem1.2} in \S4, and finish in \S5 by applying the 
Hardy-Littlewood method to prove Theorem \ref{theorem1.3}.

\par Our basic parameter is $P$, a sufficiently large positive number. In this paper, implicit 
constants in Vinogradov's notation $\ll$ and $\gg$ may depend on $s$, $r$ and $\eps$, as 
well as ambient coefficients. Whenever $\eps$ appears in a statement, either implicitly or 
explicitly, we assert that the statement holds for each $\eps>0$. We employ the 
convention that whenever $G:[0,1)^k\rightarrow \dbC$ is integrable, then
$$\oint G(\bfalp)\d\bfalp =\int_{[0,1)^k}G(\bfalp)\d\bfalp .$$
Here and elsewhere, we use vector notation in the natural way. Finally, we write $e(z)$ 
for $e^{2\pi iz}$.\par 

The authors are very grateful to a referee for identifying obscurities in the original 
version of this paper. In the current version, the treatment of the central mean value 
estimate in \S2, though somewhat longer, is both more explicit and considerably simpler in 
detail.

\section{Auxiliary equations} In this section we establish near-optimal mean value 
estimates for certain products of cubic Weyl sums. The formal coefficient matrices 
associated with these exponential sums have repeated columns, with multiplicities $2$ and 
$4$, and so would appear to be rather special. However, it transpires that this structure 
enables us to accommodate systems of cubic equations quite far from being in general 
position, and thus our principal conclusions are more flexible than the corresponding 
estimates of our earlier works \cite{BW2002, BW2014}.\par

We begin by describing the matrices important in our arguments. For $0\le l\le n$, 
consider natural numbers $r_l,s_l$ and  $r_l\times s_l$ matrices $C_l$ having non-zero 
columns. Let $\text{diag}(C_0,C_1,\ldots ,C_n)$ be the conventional diagonal block matrix 
with the upper left hand corner of $C_l$ sited at $(i_l,j_l)$. For $0\le l\le n$, append a row 
to the bottom of the matrix $C_l$, giving an $(r_l+1)\times s_l$ matrix $C'_l$. Next, 
consider the matrix $C^\dagger$ obtained from $\text{diag}(C_0,C_1,\ldots ,C_n)$ by 
replacing  $C_l$ by $C'_l$ for $0\le l\le n$, with the upper left hand corner of $C'_l$ still 
sited at $(i_l,j_l)$. We refer to this new matrix $C^\dagger$ as being a {\it linked block 
matrix}. It has additional entries by comparison to $\text{diag}(C_0,C_1,\ldots ,C_n)$, 
with the property that adjacent blocks are glued together by a shared row sited at index 
$i_l$, for $1\le l\le n$.\par

We next describe the special linked block matrices relevant to our discussion. Let $I_k$ 
denote the $k\times k$ identity matrix, and write ${\mathbf 0}$ for the zero row 
vector with $k$ components. We introduce the block matrices
$$I_k^*=\left( \begin{matrix}I_k\\{\mathbf 0}\end{matrix}\right) \quad \text{and}
\quad I_k^+=\left( \begin{matrix}{\mathbf 0}\\ I_k\\ {\mathbf 0}
\end{matrix}\right) .$$
When $n\ge 0$, $r\ge t\ge 2$ and $\ome\in \{0,1\}$, we consider fixed positive integers 
$\lam_l$, and matrices $M_l$ of format
$$\begin{cases} t\times (t+\ome),&\text{when $l=0$,}\\
r\times r,&\text{when $1\le l\le n$,}\end{cases}$$
having the property that every one of their square minors is non-singular. For ease of 
reference, we think of $M_l$ as the block matrix $(\bfm_l,B'_l)$, where $\bfm_l$ denotes 
the first column of $M_l$. Associated with each of these matrices, we consider the block 
matrices
$$A'_l=\begin{cases}(\lam_l I_{t-1}^*,\bfm_0),&\text{when $l=0$},\\
(\lam_l I^+_{r-2},\bfm_l),&\text{when $1\le l\le n$.}\end{cases}$$
Viewing the matrices $A'_l$ and $B'_l$ as examples of the matrices $C'_l$ introduced in 
the previous paragraph, we form the linked block matrices $A^\dagger$ and 
$B^\dagger$. We refer to the block matrix $D=(A^\dagger,B^\dagger)$ as an 
{\it auxiliary matrix of type} $(n,r,t)_\ome$, and write $D=(d_{ij})$. Put
\begin{equation}\label{2.1bw}
R=n(r-1)+t\quad \text{and}\quad S=2R-1+\ome .
\end{equation}
Then we see that $A^\dagger$ and $B^\dagger$ have respective formats $R\times R$ 
and $R\times (R-1+\ome)$, whilst $D$ has format $R\times S$.

To illustrate this definition, we note that all the square minors of the matrix
$${\tiny{\left( \arraycolsep=1.6pt\begin{array}{*{36}c}
7&5&6&3\\
7&1&4&8\\
9&4&5&7\\
6&3&3&8
\end{array}\right) }}$$
are non-singular, as the reader may care to verify, and hence\footnote{We 
adopt the convention that zero entries in a matrix are left blank.} 
$${\tiny{\left( \arraycolsep=1.6pt\begin{array}{*{36}c}
8&&&7&&&&&&&&&&5&6&3&&&&&&&&&&\\
&8&&7&&&&&&&&&&1&4&8&&&&&&&&&&\\
&&8&9&&&&&&&&&&4&5&7&&&&&&&&&&\\
&&&6&&&7&&&&&&&3&3&8&5&6&3&&\\
&&&&8&&7&&&&&&&&&&1&4&8&&&&&\\
&&&&&8&9&&&&&&&&&&4&5&7&&&&&\\
&&&&&&6&&&7&&&&&&&3&3&8&5&6&3\\
&&&&&&&8&&7&&&&&&&&&&1&4&8&&&\\
&&&&&&&&8&9&&&&&&&&&&4&5&7&&&\\
&&&&&&&&&6&&&7&&&&&&&3&3&8&5&6&3\\
&&&&&&&&&&8&&7&&&&&&&&&&1&4&8\\
&&&&&&&&&&&8&9&&&&&&&&&&4&5&7\\
&&&&&&&&&&&&6&&&&&&&&&&3&3&8
\end{array}\right) }}$$
is an auxiliary matrix of type $(3,4,4)_0$. Were one to delete the first row and column of 
this matrix, the result would be an auxiliary matrix of type $(3,4,3)_1$.\par

Next, consider an integral auxiliary matrix $D=(d_{ij})$ of type $(n,r,t)_\ome$, define 
$R$ and $S$ as in (\ref{2.1bw}), and define the linear forms
$$\gam_j=\sum_{i=1}^Rd_{ij}\alp_i\quad (1\le j\le S).$$
Introducing the Weyl sum
$$f(\alp)=\sum_{1\le x\le P}e(\alp x^3),$$
we define the mean value
\begin{equation}\label{2.2}
I_\ome (P;D)=\oint |f(\gam_1)\cdots f(\gam_R)|^2|f(\gam_{R+1})\cdots f(\gam_{S})|^4
\d\bfalp .
\end{equation}
Here, we use the suffix $\ome$ merely as an aide-memoire in keeping track of the type 
of the matrix $D$. We note in this context that by considering the underlying Diophantine 
system, one finds that $I_\ome (P;D)$ is unchanged by elementary row operations on 
$D$, and so in the discussion to come we may always pass to a convenient matrix row 
equivalent to $D$.\par

Before announcing our pivotal mean value estimate, we recall that Hua's lemma (see 
\cite[Lemma 2.5]{Vau1997}) shows that, for each natural number $c$, one has
\begin{equation}\label{2.4bw}
\int_0^1|f(c\tet)|^{2\nu}\d\tet \ll P^{\nu+\eps}\quad (\nu=1,2).
\end{equation}

\begin{lemma}\label{lemma2.4}
Let $D$ be an integral auxiliary matrix of type $(n,r,t)_\ome$ with $r\ge 3$. Then 
$I_\ome (P;D)\ll P^{3R-2+3\ome+\eps}$.
\end{lemma}

\begin{proof} Throughout, we assume the nomenclature for the infrastructure of the 
matrix $D$ introduced in the preamble to this lemma. We proceed by induction on 
$R\ge 2$. Since it is supposed that $t\ge 2$, it follows from (\ref{2.1bw}) that when 
$R=2$, then $n=0$ and $t=2$. In such circumstances, one has
$$I_0(P;D)=\int_0^1\!\!\int_0^1|f(\gam_1)^2f(\gam_2)^2f(\gam_3)^4|\d \alp_1\d 
\alp_2.$$
Observe that $\gam_1=\lam\alp_1$ and, since all minors of $M_0$ are non-singular, it 
follows that $\gam_3$ is linearly independent of both $\gam_1$ and $\gam_2$. Hence, by 
applying Schwarz's inequality in combination with (\ref{2.4bw}) and a change of variables, 
one obtains
$$I_0(P;D)\ll \biggl( \int_0^1|f(\tet)|^4\d\tet \biggr)^2\ll P^{4+\eps}=P^{3R-2+\eps}.$$
Meanwhile,
$$I_1(P;D)=\int_0^1\!\!\int_0^1|f(\gam_1)^2f(\gam_2)^2f(\gam_3)^4f(\gam_4)^4|
\d \alp_1\d \alp_2.$$
Since $\gam_1=\lam \alp_1$ and all minors of $M_0$ are non-singular, we may employ 
the trivial estimate $|f(\gam_2)|=O(P)$ in combination with a change of variables to 
deduce that there are fixed positive integers $a$, $b$, $c$ and $d$ for which
$$I_1(P;D)\ll P^2\int_0^1\!\!\int_0^1|f(a\tet_1)^4f(b\tet_2)^4f(c\tet_1+d\tet_2)^2|
\d\tet_1\d\tet_2 .$$
Consequently, by a pedestrian generalisation of \cite[Theorem 1]{BW2003} (see 
especially equations (6) and (7) therein), one finds that
$$I_1(P;D)\ll P^2(P^{5+\eps})=P^{3R+1+\eps}.$$
We have thus confirmed the conclusion of the lemma when $R=2$.\par

Suppose next that $R\ge 3$, and that the conclusion of the lemma holds for all auxiliary 
matrices $D$ having fewer than $R$ rows. We divide our discussion into cases 
according to the value of $\ome$.\medskip

\noindent{\it Case I: $\ome=0$}.\vskip.0cm
\noindent We first consider the situation in which the integral auxiliary matrix $D$ has $R$ 
rows and $\ome =0$. By orthogonality, one sees that $I_0(P;D)$ counts the number of 
integral solutions of the system
\begin{equation}\label{2.5bw}
\sum_{j=1}^Rd_{ij}(x_{j1}^3-x_{j2}^3)+
\sum_{j=R+1}^{S}d_{ij}(x_{j1}^3+x_{j2}^3-x_{j3}^3-x_{j4}^3)=0\quad (1\le i\le R),
\end{equation}
with $1\le x_{jl}\le P$ for each $j$ and $l$. Let $T_0$ denote the number of these 
solutions in which $x_{j1}=x_{j2}$ for $1\le j\le t-1$, and let $T_j$ denote the 
corresponding number where instead $x_{j1}\ne x_{j2}$. Then
\begin{equation}\label{2.6bw}
I_0(P;D)\le T_0+T_1+\ldots +T_{t-1}.
\end{equation}

\par An inspection of (\ref{2.5bw}) reveals that
\begin{equation}\label{2.6abw}
T_0\ll P^{t-1}J_0,
\end{equation}
where $J_0$ counts the number of integral solutions of the system
\begin{equation}\label{2.7bw}
\sum_{j=t}^Rd_{ij}(x_{j1}^3-x_{j2}^3)+
\sum_{j=R+1}^{S}d_{ij}(x_{j1}^3+x_{j2}^3-x_{j3}^3-x_{j4}^3)=0\quad (1\le i\le R),
\end{equation}
with $1\le x_{jl}\le P$ for each $j$ and $l$. We observe that the equations in 
(\ref{2.7bw}) with $1\le i\le t-1$ involve only the variables $x_{jl}$ with $j=t$ and 
$R+1\le j\le R+t-1$. The coefficient matrix associated with these equations and variables 
is the matrix $M_0^*$ obtained from $M_0$ by deleting its final row. By taking 
appropriate linear combinations of the first $t-1$ equations of (\ref{2.7bw}), 
corresponding to elementary row operations on $M_0^*$, we may therefore replace the 
equations in (\ref{2.7bw}) with $1\le i\le t-1$ by the new equations
\begin{equation}\label{2.8bw}
u_i(x_{t,1}^3-x_{t,2}^3)+v_i(x_{R+i,1}^3+x_{R+i,2}^3-x_{R+i,3}^3-x_{R+i,4}^3)=0
\quad (1\le i\le t-1),
\end{equation}
in which $u_i$ and $v_i\ne 0$ $(1\le i\le t-1)$ are suitable integers. Put $\tau=t+r-1$. 
Then adding appropriate multiples of the equations (\ref{2.8bw}) to the equation in 
(\ref{2.7bw}) with $i=t$, one finds that the latter equation may be replaced by
\begin{equation}\label{2.9bw}
u_t(x_{t1}^3-x_{t2}^3)+d_{t\tau}(x_{\tau 1}^3-x_{\tau 2}^3)+
\sum_{j=R+t}^Sd_{tj}(x_{j1}^3+x_{j2}^3-x_{j3}^3-x_{j4}^3)=0,
\end{equation}
for a suitable rational number $u_t$. The coefficient matrix $M_0^+$ associated with 
these $t$ new equations (\ref{2.8bw}) and (\ref{2.9bw}), and variables $x_{tl}$ and 
$x_{jl}$ $(R+1\le j\le R+t-1)$ has been obtained from $M_0$ by a succession of 
elementary row operations, and hence is non-singular. Since 
$\det(M_0^+)=(-1)^{t-1}u_tv_1\cdots v_{t-1}$, we therefore see that $u_t\ne 0$.\par

We now investigate the number $N_0$ of integral solutions of the system of equations 
defined by (\ref{2.9bw}) and the equations of (\ref{2.7bw}) for which $t+1\le i\le R$, 
with $1\le x_{jl}\le P$ for each $j$ and $l$. When $n=0$, the whole system reduces to 
the single equation
$$u_t(x_{t1}^3-x_{t2}^3)=0,$$
so that $N_0\ll P$. Otherwise, when $n\ge 1$, we observe that, by taking appropriate 
non-zero integral multiples of the equations, there is no loss of generality in assuming that 
$u_t=d_{ii}$ $(t+1\le i\le t+r-2)$. In this way, one finds that $N_0=I_0(P;D_1)$, where 
$D_1$ is an auxiliary matrix of type $(n-1,r,r)_0$ having $(n-1)(r-1)+r$ rows. 
Consequently, our inductive hypothesis shows that
$$I_0(P;D_1)\ll P^{3\left( (n-1)(r-1)+r\right)-2+\eps}=P^{3(R-t)+1+\eps}.$$
Then in both cases, we have $N_0\ll P^{3(R-t)+1+\eps}$.\par

Now consider any fixed solution counted by $N_0$, and consider the number $N_1$ of 
solutions $x_{jl}$ ($R+1\le j\le R+t-1$ and $1\le l\le 4$) satisfying the equations 
(\ref{2.8bw}). Since the variables $x_{t1}$ and $x_{t2}$ are fixed, it follows from 
orthogonality via the triangle inequality that
$$N_1\ll \prod_{i=1}^{t-1}\biggl( \int_0^1|f(v_i\tet_i)|^4\d\tet_i\biggr) .$$
Then we conclude from (\ref{2.4bw}) that $N_1\ll (P^{2+\eps})^{t-1}$, and hence
$$J_0\ll (P^{2+\eps})^{t-1}N_0\ll P^{3(R-t)+2(t-1)+1+t\eps}.$$
On substituting this estimate into (\ref{2.6abw}), we obtain the bound
\begin{equation}\label{2.10bw}
T_0\ll P^{3(R-t)+3(t-1)+1+\eps}=P^{3R-2+\eps}.
\end{equation}

\par We next turn to the problem of bounding $T_j$ for $1\le j\le t-1$. We restrict 
attention in the first instance to the case $j=1$, since, as will become transparent as our 
argument unfolds, the same method applies also for the remaining values of $j$. Write
\begin{equation}\label{2.9}
T(h)=\oint |f(\gam_2)\cdots f(\gam_R)|^2
|f(\gam_{R+1})\cdots f(\gam_S)|^4e(\gam_1h)\d\bfalp .
\end{equation}
Then we find by orthogonality that
$$T_1=\sum_{h\in \dbZ\setminus \{0\}}c_hT(h),$$
where $c_h$ denotes the number of integral solutions of $d_{11}(x^3-y^3)=h$, with 
$1\le x,y\le P$. An elementary divisor function estimate shows that $c_h=O(|h|^\eps)$ 
when $h\ne 0$. Since $c_h=0$ for $|h|>P^4$, one deduces from (\ref{2.9}) and a 
consideration of the underlying Diophantine system that
\begin{equation}\label{2.10}
T_1\ll P^\eps \sum_{h\in \dbZ\setminus \{0\}}T(h).
\end{equation}
The sum over $h$ on the right hand side here is bounded above by the number $N_2$ of 
solutions of the system (\ref{2.5bw}) with $2\le i\le R$ and $x_{11}=x_{12}=0$. When 
$t\ge 3$, one sees that $N_2=I_1(P;D_2)$, where $D_2$ is the auxiliary matrix of type 
$(n,r,t-1)_1$ obtained from $D$ by deleting its first row and column. Since $D_2$ has 
$R-1$ rows, it follows from the inductive hypothesis that
$$T_1\ll P^\eps I_1(P;D_2)\ll P^{3(R-1)+1+2\eps}.$$
We therefore conclude from (\ref{2.6bw}) via (\ref{2.10bw}) that when $t\ge 3$, one has
$$I_0(P;D)\ll P^{3R-2+\eps},$$
thereby confirming the inductive hypothesis for $D$.\par

It remains to handle the situation in which $t=2$. Note that since we have assumed 
$R\ge 3$, it follows that $n\ge 1$. For the sake of concision, we abbreviate 
$(\alp_2,\ldots ,\alp_R)$ to $\bfalp'$, and then define $\gam'_j(\bfalp')=
\gam_j(0,\alp_2,\ldots ,\alp_R)$. We put
$$\grF(\alp_2)=\oint |f(\gam_3')\cdots f(\gam'_R)|^2|f(\gam'_{R+2})\cdots 
f(\gam'_S)|^4\d(\alp_3,\ldots ,\alp_R),$$
and observe that, by orthogonality, one has
$$N_2=\int_0^1|f(d_{2,2}\alp_2)^2f(d_{R+1,2}\alp_2)^4|\grF(\alp_2)\d\alp_2.$$

We apply the Hardy-Littlewood method to estimate $N_2$. Denote by $\grM$ the union of 
the intervals
\begin{equation}\label{2.X}
\grM(q,a)=\{\alp\in [0,1):|q\alp-a|\le P^{-9/4}\},
\end{equation}
with $0\le a\le q\le P^{3/4}$ and $(a,q)=1$, and put $\grm=[0,1)\setminus \grM$. Let 
$c$ be a fixed non-zero integer. Then, as a special case of \cite[Lemma 3.4]{BKW2001}, 
or as a consequence of the methods of \cite[Chapter 4]{Vau1997}), one has
\begin{equation}\label{2.11}
\int_\grM |f(c\tet)|^4\d\tet \ll P^{1+\eps}.
\end{equation}
In addition, an enhanced version of Weyl's inequality (see \cite[Lemma 1]{Vau1986}) 
shows that
\begin{equation}\label{2.12}
\sup_{\tet\in \grm}|f(c\tet)|\ll P^{3/4+\eps}.
\end{equation}

\par On the one hand, it follows from (\ref{2.12}) that
$$\int_\grm |f(d_{2,2}\alp_2)^2f(d_{R+1,2}\alp_2)^4|\grF(\alp_2)\d\alp_2\ll 
P^{3+\eps}\int_0^1|f(d_{2,2}\alp_2)^2|\grF(\alp_2)\d\alp_2.$$
Put $\tau=r+1$. Then, by orthogonality, the integral on the right hand side counts the 
number of integral solutions of the system of equations given by
$$d_{22}(x_{21}^3-x_{22}^3)+d_{2\tau}(x_{j1}^3-x_{j2}^3)+
\sum_{j=R+2}^Sd_{2j}(x_{j1}^3+x_{j2}^3-x_{j3}^3-x_{j4}^3)=0$$
and
\begin{equation}\label{2.12a}
\sum_{j=3}^Rd_{ij}(x_{j1}^3-x_{j2}^3)+\sum_{j=R+2}^Sd_{ij}
(x_{j1}^3+x_{j2}^3-x_{j3}^3-x_{j4}^3)=0\quad (3\le i\le R),
\end{equation}
with $1\le x_{jl}\le P$ for each $j$ and $l$. By taking appropriate non-zero integral 
multiples of these equations, there is no loss of generality in assuming that 
$d_{22}=d_{ii}$ $(3\le i\le r)$. The coefficient matrix $D_3$ associated with these 
equations and variables arises from $D$ by deleting its first row, and the first and 
$(R+1)$-st column, and can be seen to be an auxiliary matrix of type $(n-1,r,r)_0$ 
having $R-1$ rows. It therefore follows from the inductive hypothesis that
\begin{align}
\int_\grm |f(d_{2,2}\alp)^2f(d_{R+1,2}\alp)^4|\grF(\alp)\d\alp &\ll P^{3+\eps}
I_0(P;D_3)\notag \\
&\ll P^{3+\eps}(P^{3(R-1)-2+\eps}).\label{2.13bw}
\end{align}

\par We next consider the corresponding major arc contribution. By orthogonality, the 
mean value $\grF(\alp)$ is bounded above by the number of solutions of an associated 
Diophantine system, in which each solution is counted with a 
unimodular weight depending on $\alp$. Thus we have $\grF(\alp)\le \grF(0)$. 
Consequently, it follows from (\ref{2.11}) via the trivial estimate $|f(d_{2,2}\alp)|=O(P)$ 
that
\begin{align*}
\int_\grM |f(d_{2,2}\alp)^2f(d_{R+1,2}\alp)^4|\grF(\alp)\d\alp &\ll \grF(0)P^2
\int_\grM |f(d_{R+1,2}\alp)|^4\d\alp \\
&\ll \grF(0)P^{3+\eps}.
\end{align*}
By orthogonality, the mean value $\grF(0)$ counts the number of integral solutions of the 
system (\ref{2.12a}). The coefficient matrix $D_4$ associated with these equations and 
variables arises from $D$ by deleting its first two rows, and columns $1$, $2$ and $R+1$, 
and can be seen to be an auxiliary matrix of type $(n-1,r,r-1)_1$ having $R-2$ rows. It 
therefore follows from the inductive hypothesis that
\begin{align}
\int_\grM |f(d_{2,2}\alp)^2f(d_{R+1,2}\alp)^4|\grF(\alp)\d\alp &\ll 
P^{3+\eps}I_1(P;D_4)\notag \\
&\ll P^{3+\eps}(P^{3(R-2)+1+\eps}).\label{2.14bw}
\end{align}
On combining (\ref{2.10}), (\ref{2.13bw}) and (\ref{2.14bw}), we conclude that 
when $t=2$ one has $T_1\ll P^\eps N_2\ll P^{3R-2+2\eps}$. We therefore deduce from 
(\ref{2.6bw}) and (\ref{2.10bw}) that $I_0(P;D)\ll P^{3R-2+\eps}$, confirming the 
inductive hypothesis for $D$ when $t=2$.\medskip

\noindent {\it Case II: $\ome=1$.}\vskip.0cm
\noindent We now turn to the situation in which the integral auxiliary matrix $D$ has $R$ 
rows and $\ome =1$. Observe that $\gam_{R+j}(\bfalp)$ depends only on 
$\bfalp^*=(\alp_1,\ldots ,\alp_t)$ for $1\le j\le t$. When $1\le j\le t$, we define 
$\grB_j^*$ to be the set of $t$-tuples $\bfalp^*\in [0,1)^t$ for which 
$\gam_{R+j}(\bfalp^*,{\mathbf 0})\in \grm+\dbZ$, and we define $\grB_0^*$ to be the 
complementary set of $t$-tuples $\bfalp^*\in [0,1)^t$ for which 
$\gam_{R+j}(\bfalp^*,{\mathbf 0})\not\in \grm+\dbZ$ $(1\le j\le t)$. We then put 
$\grB_j=\grB_j^*\times [0,1)^{R-t}$ $(0\le j\le t)$. Thus 
$[0,1)^R\subseteq \grB_0\cup \grB_1\cup \ldots \cup \grB_t$. When 
$\grB\subseteq [0,1)^R$, we write
$$I(\grB)=\int_\grB|f(\gam_1)\cdots f(\gam_R)|^2|f(\gam_{R+1})\cdots f(\gam_S)|^4
\d \bfalp .$$
Then it follows from (\ref{2.2}) that
\begin{equation}\label{2.21}
I_1(P;D)\le I(\grB_0)+I(\grB_1)+\ldots +I(\grB_t).
\end{equation}

\par We begin by estimating $I(\grB_1)$. It follows from (\ref{2.12}) that
$$\sup_{\bfalp\in \grB_1}|f(\gam_{R+1})|\le \sup_{\gam\in \grm}|f(\gam)|\ll 
P^{3/4+\eps},$$
and hence
$$I(\grB_1)\ll P^{3+\eps}\oint |f(\gam_1)\cdots f(\gam_R)|^2|f(\gam_{R+2})\cdots 
f(\gam_S)|^4\d \bfalp .$$
The integral on the right hand side counts the number $N_3$ of integral solutions of the 
system (\ref{2.5bw}) with $1\le x_{jl}\le P$ for each $j\ne R+1$ and $l$, but with 
$x_{R+1,l}=0$. Thus $N_3=I(P;D_5)$, where $D_5$ is the matrix obtained from $D$ by 
deleting its $(R+1)$-st column. Note that deleting a column from a matrix, all of whose 
square minors are non-singular, does not change the latter property. Hence $D_5$ is an 
auxiliary matrix of type $(n,r,t)_0$ having $R$ rows. It therefore follows from the 
inductive hypothesis that
\begin{equation}\label{2.15bw}
I(\grB_1)\ll P^{3+\eps}I_0(P;D_5)\ll P^{3+\eps}(P^{3R-2+\eps}).
\end{equation}
As indicated earlier, a symmetrical argument shows that $I(\grB_j)$ is bounded in the 
same manner for $2\le j\le t$.\par

We finish by estimating $I(\grB_0)$. Note that whenever $\bfalp\in \grB_0$, then 
$\gam_{R+j}\in \grM+\dbZ$ for $1\le j\le t$. We put
$$\grG(\bfalp^*)=\oint |f(\gam_{t+1})\cdots f(\gam_R)|^2
|f(\gam_{R+t+1})\cdots f(\gam_S)|^4\d(\alp_{t+1},\ldots ,\alp_R),$$
and apply the trivial estimate $|f(\gam_j)|\le P$ $(1\le j\le t)$. Then one finds that
\begin{equation}\label{2.16bw}
I(\grB_0)\ll P^{2t}\int_{\grB_0^*}|f(\gam_{R+1})\cdots f(\gam_{R+t})|^4\grG
(\bfalp^*)\d \bfalp^*.
\end{equation}
Observe that by orthogonality, and an argument paralleling that in the discussion following 
(\ref{2.13bw}), 
one has $\grG(\bfalp^*)\le \grG({\mathbf 0})$. Also, one sees that 
$\grG({\mathbf 0})$ counts the number of integral solutions of the system (\ref{2.5bw}) 
for $t+1\le i\le R$, with $1\le x_{jl}\le P$ for $t+1\le j\le R$ and $R+t+1\le j\le S$, and 
with the remaining variables $0$. Thus $\grG({\mathbf 0})=I_1(P;D_6)$, where $D_6$ is 
the matrix obtained from $D$ by deleting its first $t$ rows and columns $j$ with 
$1\le j\le t$ 
and $R+1\le j\le R+t$. Hence $D_6$ is an auxiliary matrix of type $(n-1,r,r-1)_1$ having 
$R-t$ rows. It therefore follows from the inductive hypothesis that 
$\grG({\mathbf 0})\ll P^{3(R-t)+1+\eps}$. By substituting this estimate into 
(\ref{2.16bw}) and making an appropriate change of variables justified by the 
non-singularity of the matrix $B_0'$, it follows that
$$I(\grB_0)\ll P^{2t}\grG({\mathbf 0})\int_{\grM^t}|f(\tet_1)\cdots f(\tet_t)|^4\d\bftet .
$$
An application of (\ref{2.11}) therefore yields
$$I(\grB_0)\ll P^{3R-t+1+\eps}(P^{1+\eps})^t\ll P^{3R+1+(t+1)\eps}.$$
In combination with (\ref{2.15bw}), and its generalisations estimating $I(\grB_j)$ for 
$2\le j\le t$, we conclude from (\ref{2.21}) that $I_1(P;D)\ll P^{3R+1+\eps}$. This 
confirms the inductive hypothesis for $D$ when $\ome=1$, completing the proof of the 
lemma.
\end{proof}

By a modification of the argument of the proof of Lemma 
\ref{lemma2.4}, one may handle also the case $r=2$. However, we are able 
to establish all of the conclusions recorded in the introduction without appealing to this 
special case.

\section{Complification} We now employ a recursive {\it complification} argument, in 
which, at each step, mean values associated with $R$ equations are estimated in terms of 
a mean value associated with $2R-1$ equations. In this way, we are able to apply the 
estimates supplied by Lemma \ref{lemma2.4} to obtain powerful estimates for suitable 
mixed moments of order $2R$ of generating functions associated with sums of three 
cubes. We begin with a lemma concerning highly non-singular matrices.

\begin{lemma}\label{lemma4.1}
Let $A=(A_1,A_2)$ be a block matrix in which $A_1$ and $A_2$ are each of format 
$r\times r$. Then $A$ is highly non-singular if and only if $A_1$ and $A_2$ are 
non-singular, and all square minors of $A_1^{-1}A_2$ are non-singular.
\end{lemma}

\begin{proof} The non-singularity condition on $A_1$ and $A_2$ is immediate from the 
definition of what it means to be highly non-singular. Thus, by applying elementary row 
operations, it suffices to consider the situation with $A=(I_r,A_1^{-1}A_2)$. The matrix 
$A$ is highly non-singular if and only if all collections of $r$ of its columns are linearly 
independent. Given any $l\times l$ minor $M$ of $A_1^{-1}A_2$, inhabiting the columns 
$\bfv_1,\ldots ,\bfv_l$, say, one can select a complementary set of columns 
$\bfe_1,\ldots ,\bfe_{r-l}$ from $I_r$ in such a manner that
$$\det(M)=\pm \det (\bfv_1,\ldots ,\bfv_l,\bfe_1,\ldots ,\bfe_{r-l}).$$
Then all collections of $r$ of the columns of $A$ are linearly independent if and only if 
$\det(M)\ne 0$ for all square minors $M$ of $A_1^{-1}A_2$, as claimed.
\end{proof}

We next prepare the cast of generating functions needed to describe the complification 
process. With the needs of \S5 in mind, we proceed in slightly greater generality than 
demanded by the proofs of Theorems \ref{theorem1.1} and \ref{theorem1.2}. Let 
$\sig\in [0,1)$. When $2\le Z\le P$, we put
$$\calA(P,Z)=\{n\in [1,P]\cap \dbZ: \text{$p$ prime and $p|n$} \Rightarrow p\le Z\},$$
and introduce the exponential sums
\begin{equation}\label{3.Y}
f_0(\alp)=\sum_{\sig P<x\le P}e(\alp x^3)\quad \text{and}\quad g(\alp)=
\sum_{\substack{\sig P<x\le P\\ x\in \calA(P,P^\eta)}}e(\alp x^3).
\end{equation}
We then take $\phi_1(\alp)=f_0(\alp)^2$ and $\phi_2(\alp)=g(\alp)^2$, and write
\begin{equation}\label{3.1}
F_l(\alp)=f_0(\alp)\phi_l(\alp) \quad \text{and}\quad \Phi_l(\alp)=f_0(\alp)^2\phi_l(\alp)
\quad (l=1,2).
\end{equation}
Finally, for the sake of convenience, we put
\begin{equation}\label{3.2}
\nu_1=\tfrac{1}{2}\quad \text{and}\quad \nu_2=3\xi=(\sqrt{2833}-43)/41.
\end{equation}

\begin{lemma}\label{lemma3.1} When $\eta>0$ is sufficiently small, one has
$$\int_0^1|F_l(\alp)|^2\d\alp \ll P^{3+\nu_l+\eps}\quad (l=1,2).$$
\end{lemma}

\begin{proof} When $l=1$, this is an immediate consequence of Hua's lemma (see 
\cite[Lemma 2.5]{Vau1997}) in combination with Schwarz's inequality. When $l=2$, 
meanwhile, this follows from \cite[Theorem 1.2]{Woo2000} by considering the 
underlying Diophantine equations.\end{proof}

Next, let $n$ and $r$ be non-negative integers with $r\ge 2$, and write $R=n(r-1)$. Let 
$\Lam=(\lam_{i,j})$ be an integral $(R+1)\times (2R+2)$ matrix, write $\bflam_j$ for the 
column vector $(\lam_{i,j})_{1\le i\le R+1}$, and define $\bflam_j^\flat$ to be the 
column vector $(\lam_{R+2-i,j})_{1\le i\le R+1}$ in which the entries of $\bflam_j$ are 
flipped upside-down. Also, let
\begin{equation}\label{3.3}
\bet_j(\bfalp)=\sum_{i=1}^{R+1}\lam_{i,j}\alp_i\quad (0\le j\le 2R+1).
\end{equation}
We say that the matrix $\Lam$ is {\it adjuvant of type} $(n,r)$ when the column 
vectors $\bflam_0,\bflam_1,\ldots ,\bflam_R$ and 
$\bflam_{R+2},\bflam_{R+3},\ldots ,\bflam_{2R+1}$, respectively, may be permuted to 
form matrices $A^\dagger$ and $B^\dagger$ having the property that the block matrix 
$(A^\dagger,B^\dagger)$ is auxiliary of type $(n-1,r,r)_0$, and also the same property 
holds for the respective column vectors 
$\bflam^\flat_{R+1},\bflam^\flat_R,\ldots ,\bflam^\flat_1$ and 
$\bflam^\flat_{2R+1},\bflam^\flat_{2R},\ldots ,\bflam^\flat_{R+2}$. We also adopt 
the convention that
$$\bfphi^{(l)}_{a,b}(\bfbet)=\prod_{j=a}^b\phi_l(\bet_j)\quad \text{and}\quad 
\bfPhi^{(l)}_{a,b}(\bfbet)=\prod_{j=a}^b\Phi_l(\bet_j).$$
We then introduce the mean value
\begin{equation}\label{3.4bw}
J_l(P;\Lam)=\oint |F_l(\bet_0)\bfphi^{(l)}_{1,R}(\bfbet)F_l(\bet_{R+1})
\bfPhi^{(l)}_{R+2,2R+1}(\bfbet)|\d\bfalp .
\end{equation}
Finally, we fix $\eta>0$ to be sufficiently small in the context of Lemma \ref{lemma3.1}.

\begin{lemma}\label{lemma3.2} Suppose that $\Lam$ is an integral adjuvant matrix of 
type $(n,r)$. Then there exists an integral adjuvant matrix $\Lam^*$ of type $(2n,r)$ for 
which
$$J_l(P;\Lam)\ll (P^{3+\nu_l+\eps})^{1/2}J_l(P;\Lam^*)^{1/2}\quad (l=1,2).$$
\end{lemma}

\begin{proof} We fix $l\in \{1,2\}$, and for the sake of concision suppress mention of $l$ 
in our notation. Define the linear forms $\bet_j$ as in (\ref{3.3}). Since the matrix $\Lam$ 
is adjuvant, one may suppose that $\bet_{R+1}=\lam_{R+1,R+1}\alp_{R+1}$ with 
$\lam_{R+1,R+1}\ne 0$. Define 
$$T(P;\Lam)=\int_0^1\left( \oint |F(\bet_0)\bfphi_{1,R}(\bfbet)\bfPhi_{R+2,2R+1}
(\bfbet)|\d{\widehat \bfalp}_R\right)^2\d\alp_{R+1},$$
where ${\widehat \bfalp}_R$ denotes $(\alp_1,\ldots,\alp_R)$. Then 
Schwarz's inequality conveys us from (\ref{3.4bw}) to the bound
\begin{equation}\label{3.5}
J(P;\Lam)\le \biggl( \int_0^1|F(\bet_{R+1})|^2\d\alp_{R+1}\biggr)^{1/2}
T(P;\Lam)^{1/2}.
\end{equation}
Define $\bet_j^*=\bet_j^*(\bfalphat_{2R+1})$ by
$$\bet_j^*=\begin{cases}
\bet_j(\alp_1,\ldots ,\alp_{R+1}),&\text{when $0\le j\le R$,}\\
\bet_{2R+1-j}(\alp_{2R+1},\ldots ,\alp_{R+1}),&\text{when $R+1\le j\le 2R+1$,}\\
\bet_{j-R}(\alp_1,\ldots ,\alp_{R+1}),&\text{when $2R+2\le j\le 3R+1$,}\\
\bet_{5R+3-j}(\alp_{2R+1},\ldots ,\alp_{R+1}),&\text{when $3R+2\le j\le 4R+1$.}
\end{cases}$$
Then, by expanding the square inside the outermost integration, we see that
$$T(P;\Lam)=\oint |F(\bet_0^*)\bfphi_{1,2R}(\bfbet^*)
F(\bet_{2R+1}^*)\bfPhi_{2R+2,4R+1}(\bfbet^*)|\d{\widehat \bfalp}_{2R+1}.$$
The integral $(2R+1)\times (4R+2)$ matrix $\Lam^*=(\lam_{ij}^*)$ defining the linear 
forms $\bet_0^*,\ldots ,\bet_{4R+1}^*$ is adjuvant of type $(2n,r)$.\par

Write $\Lam^*$ in block form $(A^*,B^*)$ with $A^*$ and $B^*$ having $2R+2$ and 
$2R$ columns, respectively. It may be illuminating to note that one may permute the 
columns of the matrix $B^*$ to form a linked block matrix $(B^*)^\dagger$ built from 
two blocks, with upper left hand block $B$ and lower right hand block $B^\flat$, in which 
$B^\flat$ denotes the matrix $B$ rotated through $180^\circ$. Likewise, one sees that 
the columns of the matrix $A^*$ may be permuted to form a linked block matrix 
$(A^*)^\dagger$ built in similar manner, but with upper left hand block $A_1$, where 
$A_1$ denotes the matrix $A$ with final column deleted, and with lower right hand block 
$A_1^\flat$, in the sense described.

Thus we conclude that $T(P;\Lam)=J(P;\Lam^*)$. The conclusion of the lemma therefore 
follows from (\ref{3.5}), since Lemma \ref{lemma3.1} supplies the estimate
$$\int_0^1 |F(\bet_{R+1})|^2\d\alp_{R+1}\ll P^{3+\nu+\eps}.$$
\end{proof}

Consider an $r\times 2r$ integral matrix $C=(c_{ij})$, write $\bfc_j$ for the column 
vector $(c_{ij})_{1\le i\le r}$, and put
$$\gam_j=\sum_{i=1}^rc_{ij}\alp_i\quad (0\le j\le 2r-1).$$
Also, write
\begin{equation}\label{3.7}
K_l(P;C)=\oint |F_l(\gam_0)\cdots F_l(\gam_{2r-1})|\d\bfalp \quad (l=1,2).
\end{equation}
We divide the proof of the next theorem according to whether $r\ge 3$ 
or $r=2$.

\begin{theorem}\label{theorem3.3}
Suppose that $r\ge 2$, and that the $r\times 2r$ integral matrix $C$ is highly non-singular. 
Then $K_l(P;C)\ll P^{3r+\nu_l+\eps}$ $(l=1,2)$.
\end{theorem}

\begin{proof}[Proof when $r\ge 3$] We again suppress mention of $l$ in our notation 
within this proof. We begin by applying Schwarz's inequality to $K(P;C)$, showing that
\begin{equation}\label{3.z1}
K(P;C)\le K^{(1)}(P;C)^{1/2}K^{(2)}(P;C)^{1/2},
\end{equation}
where
$$K^{(1)}(P;C)=\oint |F(\gam_0)\bfphi_{1,r-2}(\bfgam)F(\gam_{r-1})\phi(\gam_r)
\bfPhi_{r+1,2r-1}(\bfgam)|\d\bfalp $$
and
$$K^{(2)}(P;C)=\oint |F(\gam_0)\bfphi_{r+1,2r-2}(\bfgam)F(\gam_{r-1})
\phi(\gam_{2r-1})\Phi(\gam_r)\bfPhi_{1,r-2}(\bfgam)|\d\bfalp .$$
The coefficient matrix associated with the linear forms
$$\gam_0,\gam_{r+1},\ldots ,\gam_{2r-2},\gam_{r-1},\gam_{2r-1},\gam_r,\gam_1,
\ldots ,\gam_{r-2}$$
is obtained by permuting the columns of $C$, and hence is highly non-singular. We may 
therefore confine our attention to $K^{(1)}(P;C)$.\par

Write $C$ in block form $(A, B)$, where both $A$ and $B$ are $r\times r$ integral 
matrices, noting that the highly non-singular property of $C$ ensures that both $A$ and 
$B$ are non-singular. Note also that Lemma \ref{lemma4.1} shows that every square 
minor of $A^{-1}B$ is non-singular. It is convenient to put
$$\bfgam^{(1)}=(\gam_0,\ldots ,\gam_{r-1})^T\quad \text{and}\quad 
\bfgam^{(2)}=(\gam_r,\ldots ,\gam_{2r-1})^T.$$
We then have
$$\bfgam^{(1)}=A^T\bfalp\quad \text{and}\quad \bfgam^{(2)}=B^T\bfalp .$$ 
Let $\Del=|\!\det A|$. We substitute $\bftet=\Del^{-1}A^T\bfalp$, so that 
$$\bfgam^{(1)}=\Del\bftet\quad \text{and}\quad 
\bfgam^{(2)}=\Del(A^{-1}B)^T\bftet,$$
and then define the linear forms $\bet_j(\tet_1,\ldots ,\tet_r)\in \dbZ[\bftet]$ by means 
of the relation $\bfbet=\Del(A^{-1}B)^T\bftet$, in which 
$\bfbet=(\bet_{r+1},\ldots ,\bet_{2r})^T$. Since the underlying exponential sums are
periodic with period $1$, we may apply the transformation formula to conclude that
$$K^{(1)}(P;C)=\oint |F(\Del \tet_1)\bfphi_{2,r-1}(\Del \bftet)F(\Del \tet_r)
\phi(\bet_{r+1})\bfPhi_{r+2,2r}(\bfbet)|\d\bftet .$$
The matrix of coefficients of the linear forms defining this mean value, namely 
$$\Del\tet_1,\ldots ,\Del\tet_r,\bet_{r+1}(\bftet),\ldots ,\bet_{2r}(\bftet),$$
is given by $(\Del I_r, \Del(A^{-1}B)^T)$, which, in view of Lemma \ref{lemma4.1}, is 
highly non-singular. In particular, all the square minors of $\Del(A^{-1}B)^T$ are 
non-singular.\par

Define
$$T(P;C)=\int_0^1\biggl( \oint |F(\Del \tet_1)\bfphi_{2,r-1}(\Del\bftet) \phi(\bet_{r+1})
\bfPhi_{r+2,2r}(\bfbet)|\d\bftethat_{r-1}\biggr)^2\d\tet_r.$$
Then by Schwarz's inequality, one finds that
\begin{equation}\label{3.z2}
K^{(1)}(P;C)\le \biggl( \int_0^1|F(\Del \tet_r)|^2\d\tet_r\biggr)^{1/2}T(P;C)^{1/2}.
\end{equation}
By expanding the square in the definition of $T(P;C)$, we see that
$$T(P;C)=\oint |F(\bet_0^*)\bfphi_{1,2r-2}(\bfbet^*)F(\bet_{2r-1}^*)\bfPhi_{2r,4r-3}
(\bfbet^*)|\d\bftethat_{2r-1},$$
where $\bet_j^*=\bet_j^*(\bftet)$ is defined by
$$\bet_j^*=\begin{cases}
\Del \tet_{j+1},&\text{when $0\le j\le 2r-3$ and $j\ne r-1$,}\\
\bet_{r+1}(\tet_1,\ldots ,\tet_r),&\text{when $j=r-1$,}\\
\bet_{r+1}(\tet_{2r-1},\ldots ,\tet_r),&\text{when $j=2r-2$,}\\
\Del \tet_{2r-1},&\text{when $j=2r-1$,}\\
\bet_{j-r+2}(\tet_1,\ldots ,\tet_r),&\text{when $2r\le j\le 3r-2$,}\\
\bet_{5r-1-j}(\tet_{2r-1},\ldots ,\tet_r),&\text{when $3r-1\le j\le 4r-3$.}
\end{cases}$$
It is apparent that the matrix of coefficients $C'$ of the linear forms
$$\bet_0^*(\bftet),\ldots ,\bfbet_{4r-3}^*(\bftet)$$
is an integral adjuvant matrix of type $(2,r)$. Thus, in the notation introduced in 
(\ref{3.4bw}), we see that $T(P;C)=J(P;C')$. By virtue of the conclusion of Lemma 
\ref{lemma3.1}, we therefore infer from (\ref{3.z1}) and (\ref{3.z2}) that there exist 
integral adjuvant matrices $C^{(1)}$ and $C^{(2)}$ of type $(2,r)$ for which
\begin{align}
K(P;C)&\ll (P^{3+\nu+\eps})^{1/2}J(P;C^{(1)})^{1/4}J(P;C^{(2)})^{1/4}\notag \\
&\ll (P^{3+\nu+\eps})^{1/2}\max_{\iota=1,2}J(P;C^{(\iota)})^{1/2}.\label{3.z3}
\end{align}
By symmetry, there is no loss of generality in supposing that the maximum on the right 
hand side occurs with $\iota=1$.\par

We put $B_1=C^{(1)}$, and show by induction that for each natural number $m$, there 
exists an integral adjuvant matrix $B_m$ of type $(2^m,r)$ for which
\begin{equation}\label{3.8}
K(P;C)\ll (P^{3+\nu+\eps})^{1-2^{-m}}J(P;B_m)^{2^{-m}}.
\end{equation}
This bound holds when $m=1$ as a trivial consequence of (\ref{3.z3}). 
Suppose then that the estimate (\ref{3.8}) holds for $1\le m\le M$. By applying Lemma 
\ref{lemma3.2}, we see that there exists an integral adjuvant matrix $B_{M+1}$ of type 
$(2^{M+1},r)$ with
$$J(P;B_M)\ll (P^{3+\nu+\eps})^{1/2}J(P;B_{M+1})^{1/2}.$$
Substituting this estimate into the case $m=M$ of (\ref{3.8}), one confirms that the 
bound (\ref{3.8}) holds with $m=M+1$. The bound (\ref{3.8}) consequently follows for all 
$m$ by induction.\par

We now apply the bound just established. Let $\del$ be any small positive number, and 
choose $m$ large enough that $2^{1-m}(2-\nu)<\del$. We have shown that an integral 
adjuvant matrix $B_m=(b_{ij})$ of type $(2^m,r)$ exists for which (\ref{3.8}) holds. The 
matrix $B_m$ is of format $(R+1)\times (2R+2)$, where $R=2^m(r-1)$. In view of 
(\ref{3.4bw}), together with the trivial estimates $|f(\alp)|\le P$ and $|g(\alp)|\le P$, we 
find that
$$J(P;B_m)\ll P^4\oint |\bfphi_{0,R}(\bfbet)\bfPhi_{R+2,2R+1}(\bfbet)|\d
\bfalp.$$
The matrix of coefficients associated with a suitable permutation of the linear forms
$$\bet_0(\bfalp),\bet_1(\bfalp),\ldots ,\bet_R(\bfalp),\bet_{R+2}(\bfalp),\ldots ,
\bet_{2R+1}(\bfalp),$$
is auxiliary of type $(2^m-1,r,r)_0$. By orthogonality, a consideration of the underlying 
Diophantine equations shows that $J(P;B_m)\ll P^4I_0(P;B_m)$, and hence we deduce 
from Lemma \ref{lemma2.4} that
$$J(P;B_m)\ll P^4\left(P^{3(R+1)-2+\eps}\right)=P^{3(2^m(r-1))+5+\eps}.$$
By substituting the estimate just obtained into (\ref{3.8}), we conclude that
\begin{align*}
K(P;C)&\ll \left(P^{3+\nu+\eps}\right)^{1-2^{-m}}\left( P^{3(2^m(r-1))+5+\eps}
\right)^{2^{-m}}\\
&=P^{3r+\nu+(2-\nu)2^{-m}+\eps}.
\end{align*}
In view of our assumed upper bound $2^{1-m}(2-\nu)<\del$, one therefore finds that for 
each $\eps'>0$, one has $K(P;C)\ll P^{3r+\nu+\del/2+\eps'}$. The conclusion of the 
theorem follows on taking $\del=\eps$ and $\eps'=\tfrac{1}{2}\eps$.
\end{proof}

\begin{proof}[Proof when $r=2$] An application of the elementary inequality 
$|z_1\cdots z_n|\le |z_1|^n+\ldots +|z_n|^n$ yields
$$F_l(\gam_0)\cdots F_l(\gam_3)\ll \sum_{0\le a<b<c\le 3}
|F_l(\gam_a)F_l(\gam_b)F_l(\gam_c)|^{4/3},$$
and hence there exist integers $a$, $b$ and $c$ with $0\le a<b<c\le 3$ for which
\begin{equation}\label{3.z}
K_l(P;C)\ll \oint |F_l(\gam_a)F_l(\gam_b)F_l(\gam_c)|^{4/3}\d\bfalp .
\end{equation}
It is convenient to define
$$\Ome_h(P;C)=\oint |h(\gam_a)h(\gam_b)h(\gam_c)|^4\d\bfalp ,$$
with $h$ taken to be either $f_0$ or $g$. It is immediate from 
\cite[Theorem 1.8]{KW2010} that
\begin{equation}\label{3.zz}
\Ome_g(P;C)\ll P^{6+(6\nu_2-1)/4+\eps}.
\end{equation}
The argument of the proof of the latter theorem also readily yields the estimate 
$\Ome_{f_0}(P;C)\ll P^{6+\nu_1+\eps}$. In order to see this, one observes that the 
bound
$$\int_0^1|f_0(\alp)|^6\d\alp \ll P^{3+\nu_1+\eps},$$
stemming from Hua's lemma (see \cite[Lemma 2.5]{Vau1997}), can be substituted for 
the bound
$$\int_0^1|f_0(\alp)^2g(\alp)^4|\d\alp \ll P^{3+\nu_2+\eps}$$
underlying the proof of \cite[Theorem 1.8]{KW2010}. In this way, one finds as in 
\cite[equation (4.10)]{KW2010} that
$$\Ome_{f_0}(P;C)\ll P^{23/4+3\nu_1/2+\eps}=P^{6+\nu_1+\eps}.$$

\par In the above notation, when $l=1$, we now infer from the bound (\ref{3.z}) that
$$K_1(P;C)\ll \Ome_{f_0}(P;C)\ll P^{6+\nu_1+\eps}.$$
Also, applying H\"older's inequality to (\ref{3.z}), we obtain via (\ref{3.zz}) the estimate
$$K_2(P;C)\ll \Ome_{f_0}(P;C)^{1/3}\Ome_g(P;C)^{2/3}\ll P^{6+\nu_2+\eps},$$
on observing that
$$\frac{1}{3}\nu_1+\frac{2}{3}\left( \frac{6\nu_2-1}{4} \right)=\nu_2.$$
This completes the proof of Theorem \ref{theorem3.3} in the case $r=2$.
\end{proof}

\section{Correlation estimates} We apply Theorem \ref{theorem3.3} in this section to 
provide estimates for the correlation sums $\Xi_{s,\eta}(N;A;\bfh)$. By reference to 
(\ref{1.1}) and (\ref{1.2}), we see that when $A\in \dbZ^{r\times 2r}$ is a highly 
non-singular matrix, and $h_i\in \dbN\cup \{ 0\}$, then $\Xi_{2r,\eta}(N;A;\bfh)$ counts 
the number of integral solutions of the system
\begin{equation}\label{4.1}
X_j=\Lam_j(\bfn)\quad (1\le j\le 2r),
\end{equation}
with $\bfn\in \calP(N)$, in which $X_j=x_j^3+y_j^3+z_j^3-h_j$ and 
$x_j,y_j,z_j\in \dbN$, and none of the prime divisors of $y_jz_j$ exceed 
$(X_j+h_j)^{\eta/3}$. Since $X_j+h_j$ is no larger than $CN$, for a suitable positive 
constant $C$ depending at most on the coefficients of the $\Lam_j$, one sees that 
$x_j,y_j,z_j$ are each bounded above by $P=(CN)^{1/3}$.\par

The system (\ref{4.1}) may be written in the shape $A^T\bfn=\bfX$. It is convenient to 
consider a block matrix decomposition of $A$, say $A=(A_1,A_2)$ with $A_1$ and $A_2$ 
each $r\times r$ matrices, and also to write
$$\bfX=\left( \begin{matrix} \bfX_1\\ \bfX_2\end{matrix}\right) ,$$
with $\bfX_1$ and $\bfX_2$ each $r$-dimensional column vectors. Thus 
$\bfX_i=A_i^T\bfn$ for $i=1,2$. Since $A$ is highly non-singular, the matrices $A_1$ and 
$A_2$ are necessarily invertible, and we deduce that
$$(A_1^{-1})^T\bfX_1=\bfn=(A_2^{-1})^T\bfX_2.$$
Thus we find that $B'\bfX={\mathbf 0}$, where
$$B'=\left((A_1^{-1})^T,-(A_2^{-1})^T\right).$$
By applying Lemma \ref{lemma4.1}, one sees that the matrix $B'$ is highly non-singular if 
and only if $(A_1^{-1})^T$ and $(A_2^{-1})^T$ are non-singular, and all the square 
minors of $A_1^T(A_2^{-1})^T=(A_2^{-1}A_1)^T$ are non-singular. The non-singularity 
of $(A_1^{-1})^T$ and $(A_2^{-1})^T$ is immediate from that of $A_1$ and $A_2$. 
Likewise, the non-singularity of the square minors of $(A_2^{-1}A_1)^T$ is equivalent to 
the non-singularity of the square minors of $A_2^{-1}A_1$, which is a consequence of 
the highly non-singular nature of the block matrix $(A_2,A_1)$, again by Lemma 
\ref{lemma4.1}. We hence conclude that $B'$ is highly non-singular. Finally, we take 
$\lam$ to be the least natural number with the property that $\lam B'$ has integral 
entries, and define the matrix $B=(b_{ij})$ by putting $B=\lam B'$.\par
 
At this point, we have established that $\Xi_{2r,\eta}(N;A;\bfh)$ is bounded above by the 
number of solutions of the system of equations
$$\sum_{j=1}^{2r}b_{ij}(x_j^3+y_j^3+z_j^3)=H_i\quad (1\le i\le r),$$
with $1\le x_j\le P$ and $y_j,z_j\in \calA(P,P^\eta)$ $(1\le j\le 2r)$, in which
$$H_i=\sum_{j=1}^{2r}b_{ij}h_j.$$
Define
$$\bet_j=\sum_{i=1}^rb_{ij}\alp_i\quad (1\le j\le 2r).$$
Making use of the notation (\ref{3.1}) with $\sig$ taken implicitly to be $0$, it therefore 
follows from orthogonality that
$$\Xi_{2r,\eta}(N;A;\bfh)\le \oint F_l(\bet_1)\cdots F_l(\bet_{2r})
e(-\bfalp \cdot \bfH)\d\bfalp \quad (l=1,2).$$
We note here that one should view $\eta$ as being $1$ in the case $l=1$, and when 
$l=2$ view $\eta $ as being a positive number sufficiently small in terms of $\eps$. An 
application of the triangle inequality in conjunction with Theorem \ref{theorem3.3} 
consequently reveals that $\Xi_{2r,\eta}(N;A;\bfh)\ll P^{3r+\nu_l+\eps}$ $(l=1,2)$. 
Theorems \ref{theorem1.1} and \ref{theorem1.2} follow by reference to (\ref{3.2}), 
since one has $P=O(N^{1/3})$.

\section{Systems of linear equations} We turn now to the proof of Theorem 
\ref{theorem1.3}. Let $C=(c_{ij})$ denote an integral $r\times s$ highly non-singular 
matrix with $r\ge 2$ and $s\ge 2r+1$. We define
$$\gam_j(\bfalp)=\sum_{i=1}^rc_{ij}\alp_i\quad (1\le j\le s).$$
Let $N$ be a large positive number, and 
put $P=\tfrac{1}{2}N^{1/3}$. Let $\eta$ be a positive number sufficiently small in the 
context of Lemma \ref{lemma3.1}, and let $\sig$ be a positive number sufficiently small in 
terms of $C$ and $\eta$. Recalling (\ref{3.Y}), we put $\grf(\alp)=f_0(\alp)$ and 
$\grg(\alp)=g(\alp)$, and for the sake of concision write $\grg_j=\grg(\gam_j(\bfalp))$ 
and $\grf_j=\grf(\gam_j(\bfalp))$. When $\grB\subseteq [0,1)^r$ is measurable, we then 
define
$$\calN(P;\grB)=\int_\grB \prod_{j=1}^s\grf_j\grg_j^2\d\bfalp .$$
By orthogonality, it follows from this definition that $\calN(P;[0,1)^r)$ counts the number 
of integral solutions of the system
\begin{equation}\label{5.1}
\sum_{j=1}^sc_{ij}(x_j^3+y_j^3+z_j^3)=0\quad (1\le i\le r),
\end{equation}
with $\sig P<x_j,y_j,z_j\le P$ and $y_j,z_j\in \calA(P,P^\eta)$ $(1\le j\le s)$. Hence we 
find that $\calN(P;[0,1)^r)$ counts the solutions of the system (\ref{1.6}) with each 
solution $\bfn$ counted with weight $\rho_\eta(n_1;P)\cdots \rho_\eta(n_s;P)$, in which 
$\rho_\eta(n;P)$ denotes the number of integral solutions of the equation 
$n=x^3+y^3+z^3$, with $\sig P<x,y,z\le P$ and $y,z\in \calA(P,P^\eta)$. We aim to show 
that $\calN(P;[0,1)^r)\gg (P^3)^{s-r}$.\par

In pursuit of the above objective, we apply the Hardy-Littlewood method. Write 
$L=\log \log P$, denote by $\grN$ the union of the intervals
$$\grN(q,a)=\{\alp \in [0,1):|q\alp -a|\le LP^{-3}\},$$
with $0\le a\le q\le L$ and $(a,q)=1$, and put $\grn=[0,1)\setminus \grN$. Finally, we 
introduce a multi-dimensional set of arcs. Let $Q=L^{10r}$, and define the narrow set of 
major arcs $\grP$ to be the union of the boxes
$$\grP(q,\bfa)=\{ \bfalp\in [0,1)^r:|\alp_i-a_i/q|\le QP^{-3}\ (1\le i\le r)\},$$
with $0\le a_i\le q\le Q$ $(1\le i\le r)$ and $(a_1,\ldots ,a_r,q)=1$.

\begin{lemma}\label{lemma5.1} One has $\calN(P;\grP)\gg P^{3s-3r}$.
\end{lemma}

\begin{proof} We begin by defining the auxiliary functions
$$S(q,a)=\sum_{r=1}^qe(ar^3/q)\quad \text{and}\quad v(\bet)=\int_{\sig P}^P
e(\bet \gam^3)\d\gam .$$
For $1\le j\le s$, put $S_j(q,\bfa)=S(q,\gam_j(\bfa))$ and 
$v_j(\bfbet)=v(\gam_j(\bfbet))$, and define
\begin{equation}\label{5.2}
A(q)=\underset{(q,a_1,\ldots ,a_r)=1}{\sum_{a_1=1}^q\cdots \sum_{a_r=1}^q}
q^{-3s}\prod_{j=1}^sS_j(q,\bfa)^3\quad \text{and}\quad 
V(\bfbet)=\prod_{j=1}^sv_j(\bfbet)^3.
\end{equation}
Finally, write $\calB(X)$ for $[-XP^{-3},XP^{-3}]^r$, and define
$$\grJ(X)=\int_{\calB(X)}V(\bfbet)\d\bfbet \quad \text{and}\quad \grS(X)=
\sum_{1\le q\le X}A(q).$$

\par We prove first that there exists a positive constant $\calC$ with the property that
\begin{equation}\label{5.3}
\calN(P;\grP)-\calC\grS(Q)\grJ(Q)\ll P^{3s-3r}L^{-1}.
\end{equation}
It follows from \cite[Lemma 8.5]{Woo1991} (see also \cite[Lemma 5.4]{Vau1989}) that 
there exists a positive constant $c=c(\eta)$ such that whenever 
$\bfalp \in \grP(q,\bfa)\subseteq \grP$, then
$$\grg(\gam_j(\bfalp))-cq^{-1}S_j(q,\bfa)v_j(\bfalp-\bfa/q)\ll P(\log P)^{-1/2}.$$
Under the same constraints on $\bfalp$, one finds from \cite[Theorem 4.1]{Vau1997} that
$$\grf(\gam_j(\bfalp))-q^{-1}S_j(q,\bfa)v_j(\bfalp-\bfa/q)\ll \log P.$$ 
Thus, whenever $\bfalp\in \grP(q,\bfa)\subseteq \grP$, one has
$$\prod_{j=1}^s\grf_j\grg_j^2-c^{2s}q^{-3s}\prod_{j=1}^s
S_j(q,\bfa)^3v_j(\bfalp-\bfa/q)^3\ll P^{3s}(\log P)^{-1/2}.$$
The measure of the major arcs $\grP$ is $O(Q^{2r+1}P^{-3r})$, so that on integrating 
over $\grP$, we confirm the relation (\ref{5.3}) with $\calC=c^{2s}$.\par

We next discuss the singular integral $\grJ(Q)$. By applying an argument paralleling that 
of \cite{BW2014} leading to equation (4.4) of that paper, one finds that
\begin{equation}\label{5.4}
\grJ(Q)\gg P^{3s-3r}.
\end{equation}
Here, we make use of the hypothesis that the system (\ref{1.6}) has a solution 
$\bfn\in (0,\infty)^s$, and 
hence also one with $\bfn\in (0,1)^s$. Thus, on taking $\sig$ sufficiently small, we ensure 
that a non-singular solution $\bfn$ of (\ref{1.6}) exists with $\bfn\in (2\sig,1)^s$.\par

We turn now to the singular series $\grS(Q)$. It follows from 
\cite[Theorem 4.2]{Vau1997} that whenever $(q,a)=1$, one has $S(q,a)\ll q^{2/3}$. 
Given a summand $\bfa$ in the formula for $A(q)$ provided in (\ref{5.2}), write 
$h_j=(q,\gam_j(\bfa))$. Then we find that
$$A(q)\ll q^{-s}\underset{(q,a_1,\ldots ,a_r)=1}{\sum_{a_1=1}^q\cdots 
\sum_{a_r=1}^q}h_1\cdots h_s.$$
By hypothesis, we have $s\ge 2r+1$. The proof of \cite[Lemma 23]{DL1969} is therefore 
easily modified to show that
$$A(q)\ll \underset{(u_1,\ldots ,u_r)\ll 1}{\sum_{u_1|q}\ldots \sum_{u_r|q}}
q^{r-s}(u_1\cdots u_r)^{s/r-1}\ll q^{r-s+(s-r)(1-1/r)+\eps}.$$
Thus, the series $\grS=\underset{X\rightarrow \infty}\lim \grS(X)$ is absolutely 
convergent and
$$\grS-\grS(Q)\ll \sum_{q>Q}q^{-1-1/(2r)}\ll Q^{-1/(2r)}\ll L^{-1}.$$

\par We observe in the next step that the system (\ref{5.1}) has a non-singular $p$-adic 
solution. For on taking $(x_j,y_j,z_j)=(1,-1,0)$ for each $j$, we solve (\ref{5.1}) with the 
Jacobian determinant
$$\det(3c_{ij}x_j^2)_{1\le i,j\le r}=3^r\det(c_{ij})_{1\le i,j\le r}$$
non-zero, since the first $r$ columns of $C$ are linearly independent. A modification of the 
proof of \cite[Lemma 31]{DL1969} therefore shows that $\grS>0$, whence 
$\grS(Q)=\grS+O(L^{-1})>0$. The proof of the lemma is completed by recalling 
(\ref{5.4}) and substituting into (\ref{5.3}) to obtain the lower bound
$$\calN(P;\grP)\gg P^{3s-3r}+O(P^{3s-3r}L^{-1}).$$
\end{proof}

Recall the definition of the major arcs $\grM(q,a)$ and their union $\grM$ from 
(\ref{2.X}). In order to prune a wide set of major arcs down to the narrow set $\grP$ just 
considered, we introduce the
 auxiliary sets of arcs
$$\grM_j=\{\bfalp \in [0,1)^r:\gam_j(\bfalp)\in \grM+\dbZ\},$$
and we put $\grV=\grM_1\cap\ldots \cap \grM_s$. In addition, we define 
$\grm_j=[0,1)^r\setminus \grM_j$ $(1\le j\le s)$, and write 
$\grv=[0,1)^r\setminus \grV$. Thus $\grv\subseteq \grm_1\cup \ldots \cup \grm_s$. We 
begin with an auxiliary lemma.

\begin{lemma}\label{lemma5.2} Let $\del$ be a fixed positive number. Then one has 
$$\int_\grM |\grf(\tet)\grg(\tet)^2|^{2+\del}\d\tet \ll P^{3+3\del}$$
and
$$\int_{\grM\setminus \grN}|\grf(\tet)\grg(\tet)^2|^{2+\del}\d\tet \ll 
P^{3+3\del}L^{-\del/6}.$$
\end{lemma}

\begin{proof} On applying a special case of \cite[Lemma 9]{BW2007}, we obtain the 
bound
$$\int_\grM |\grf(\tet)|^{2+\del}|\grg(\tet)|^2\d\tet \ll P^{1+\del},$$
and so the first conclusion follows on making use of a trivial estimate for $\grg(\tet)$. For 
the second inequality, one observes that the methods of \cite[Chapter 4]{Vau1997} show 
that
$$\sup_{\alp\in \grM\setminus \grN}|\grf(\tet)|\ll PL^{-1/3}.$$
Thus, on making use also of a trivial estimate for $\grg(\tet)$, one obtains in like manner 
the bound
\begin{align*}
\int_{\grM\setminus \grN}|\grf(\tet)\grg(\tet)^2|^{2+\del}\d\tet &\ll 
(PL^{-1/3})^{\del/2}P^{2+2\del}\int_{\grM\setminus \grN}|\grf(\tet)|^{2+\del/2}
|\grg(\tet)|^2\d\tet \\
&\ll (PL^{-1/3})^{\del/2}P^{3+5\del/2}.
\end{align*}
This completes the proof of the lemma.
\end{proof}

\begin{lemma}\label{lemma5.3}
One has $\calN(P;\grV\setminus \grP)\ll P^{3s-3r}(\log L)^{-1}$.
\end{lemma}

\begin{proof} Let $\bfalp\in \grV\setminus \grP$, and suppose temporarily that 
$\gam_{j_m}\in \grN+\dbZ$ for $r$ distinct indices $j_m\in [1,s]$. For each $m$ there is 
a natural number $q_m\le L$ having the property that $\|q_m\gam_{j_m}\|\le LP^{-3}$. 
With $q=q_1\cdots q_r$, one has $q\le L^r$ and $\|q\gam_{j_m}\|\le L^rP^{-3}$. Next 
eliminating between $\gam_{j_1},\ldots ,\gam_{j_r}$ in order to isolate 
$\alp_1,\ldots ,\alp_r$, one finds that there is a positive integer $\kap$, depending at 
most on $(c_{ij})$, such that $\|\kap q\alp_l\|\le L^{r+1}P^{-3}$ $(1\le l\le r)$. Since 
$\kap q\le L^{r+1}$, it follows that $\bfalp\in \grP$, yielding a contradiction to our 
hypothesis that $\bfalp\in \grV\setminus \grP$. Thus $\gam_\nu(\bfalp)\in \grn+\dbZ$ for 
at least $s-r\ge r+1$ of the suffices $\nu$ with $1\le \nu \le s$. Let $\calH$ denote the 
set of all $r$ element subsets of $\{1,2,\ldots ,s\}$, and put $H=\text{card}(\calH)$. 
Then by H\"older's inequality, we find that
\begin{equation}\label{5.5}
\calN(P;\grV\setminus \grP)\le \prod_{\bfnu\in \calH}I(\bfnu)^{1/H},
\end{equation}
where
$$I(\bfnu)=\int_{\grV\setminus \grP}\prod_{j=1}^r
|\grf_{\nu_j}\grg_{\nu_j}^2|^{s/r}\d\bfalp .$$

\par When $\bfnu\in \calH$, one finds by a change of variable that
$$I(\bfnu)\le \int_{\grM^r}\prod_{j=1}^r|\grf(\bet_j)\grg(\bet_j)^2|^{s/r}\d\bfbet ,$$
so that Lemma \ref{lemma5.2} shows that $I(\bfnu)\ll P^{3s-3r}$. Further, since there 
exists some 
$\bfnu\in \calH$ such that $\gam_{\nu_j}(\bfalp)\in \grn+\dbZ$ for $1\le j\le r$, one finds 
for this subset that one has the bound
$$I(\bfnu)\le \int_{(\grM\setminus \grN)^r}\prod_{j=1}^r
|\grf(\bet_j)\grg(\bet_j)^2|^{s/r}\d\bfbet \ll P^{3s-3r}L^{-1/6}.$$
Thus we conclude from (\ref{5.5}) that
$$\calN(P;\grV\setminus \grP)\ll P^{3s-3r}L^{-1/(6H)},$$
and the conclusion of the lemma follows.
\end{proof}

Our final task in the application of the Hardy-Littlewood method is the analysis of the 
minor arcs $\grv$.

\begin{lemma}\label{lemma5.4} There is a positive number $\del$ such that 
$\calN(P;\grv)\ll P^{3s-3r-\del}$.
\end{lemma}

\begin{proof} Since $\grv\subseteq \grm_1\cup \ldots \cup \grm_s$, the conclusion of the 
lemma follows by showing that $\calN(P;\grm_j)\ll P^{3s-3r-\del}$ for $1\le j\le s$. By 
symmetry, moreover, we may restrict attention to the case $j=s$. Suppose then that 
$\gam_s(\bfalp)\in \grm+\dbZ$. Observe that the matrix $C$ is highly non-singular, and 
thus the matrix $C'$, in which the final $s-2r$ columns of $C$ are deleted, is also highly 
non-singular. Then it follows from (\ref{3.7}) and Theorem \ref{theorem3.3} that
$$\oint \prod_{j=1}^{2r}|\grf_j\grg_j^2|\d\bfalp \ll P^{3r+\nu_2+\eps}.$$
Observe that by Weyl's inequality (see \cite[Lemma 1]{Vau1986}), one has
$$\sup_{\gam_s(\bfalp)\in \grm+\dbZ}|\grf(\gam_s(\bfalp))|\ll P^{3/4+\eps}.$$
Hence, by employing trivial estimates for $\grf_j$ and $\grg_j$ as necessary, one obtains 
the bound
\begin{align*}
\calN(P;\grm_s)&\le P^{3s-6r-1}\biggl( \sup_{\gam_s(\bfalp)\in \grm+\dbZ}|
\grf(\gam_s(\bfalp))|\biggr) \oint \prod_{j=1}^{2r}|\grf_j\grg_j^2|\d\bfalp \\
&\ll P^{3s-3r+\nu_2-1/4+\eps}.
\end{align*}
From (\ref{3.2}), we have $\nu_2<1/4$, and so the conclusion of the lemma now follows.
\end{proof}

By combining the conclusions of Lemmata \ref{lemma5.1}, \ref{lemma5.3} and 
\ref{lemma5.4}, we conclude that
\begin{equation}\label{5.6}
\calN(P)=\calN(P;\grP)+\calN(P;\grV\setminus \grP)+\calN(P;\grv)\gg P^{3s-3r}.
\end{equation}

\par Our final task is to remove the multiplicity of representations implicit in the definition 
of $\rho_\eta(n;P)$. Note that $\rho_\eta(n;P)\le \rho_\eta(n)$ for each $n\in \dbN$. It is 
useful to introduce the set
$$\calS_\tet(N)=\{ 1\le n\le N:\rho_\eta(n)>N^\tet\} .$$

\begin{lemma}\label{lemma5.5}
One has
$$\sum_{n\in \calS_\tet(N)}\rho_\eta(n)\ll N^{1+\xi-\tet+\eps}.$$
\end{lemma}

\begin{proof} In view of (\ref{1.Z}), one has
$$\sum_{n\in \calS_\tet(N)}\rho_\eta(n)<N^{-\tet}\sum_{n\in \calS_\tet(N)}
\rho_\eta(n)^2\ll N^{1+\xi-\tet+\eps},$$
and the conclusion of the lemma follows. 
\end{proof}

Let $\del$ be a positive number, and consider the number $Y_1$ of solutions of the 
system (\ref{5.1}) in which one has $\rho_\eta (x_j^3+y_j^3+z_j^3)>N^{2\xi+\del}$ 
for some index $j$ with $1\le j\le s$. Without loss of generality, one may assume that 
$j=s$. Then by orthogonality, one has
$$Y_1\ll \sum_{n\in \calS_{2\xi+\del}(N)}\rho_\eta(n)\oint \biggl( 
\prod_{j=1}^{s-1}\grf_j\grg_j^2\biggr)e(n\gam_s(\bfalp))\d\bfalp .$$
By the triangle inequality, Theorem \ref{theorem3.3} and Lemma \ref{lemma5.5}, we 
thus deduce that
\begin{align*}
Y_1&\ll N^{s-r-1+\xi+\eps}\sum_{n\in \calS_{2\xi+\del}(N)}\rho_\eta(n)\\
&\ll N^{s-r+2\xi-(2\xi+\del)+2\eps}\ll N^{s-r-\del/2}.
\end{align*}
Let $Y_0$ denote the contribution to $\calN(P)$ arising from those solutions of (\ref{5.1}) 
in which $\rho_\eta (x_j^3+y_j^3+z_j^3)\le N^{2\xi+\del}$ for all $j$. Then it follows 
from (\ref{5.6}) that
$$Y_0\gg N^{s-r}+O(N^{s-r-\del/2})\gg N^{s-r}.$$
Since $Y_0$ counts solutions $\bfn$ of (\ref{1.6}), with each solution counted with weight 
at most $\rho_\eta(n_1)\cdots \rho_\eta(n_s)\le (N^{2\xi+\del})^s$, we conclude that 
$\Ups(N)\gg N^{s-r}(N^{2\xi+\del})^{-s}$. As $\del$ may be chosen arbitrarily small, 
though positive, this completes the proof of Theorem \ref{theorem1.3}.

\bibliographystyle{amsbracket}
\providecommand{\bysame}{\leavevmode\hbox to3em{\hrulefill}\thinspace}

\end{document}